\documentclass[12pt]{amsart}

\usepackage{graphicx}        

\usepackage{amsmath,amssymb,amsthm,latexsym}
\theoremstyle{plain}
\usepackage[margin=1.45in,top=1in,bottom=1in]{geometry}
\usepackage[cmtip,all]{xy}
\usepackage{epstopdf}
\usepackage{float}
\usepackage{scalerel}

\usepackage{amsmath,amssymb,amsthm,latexsym}

\newtheorem*{theorem}{Theorem}
\newtheorem*{lemma}{Lemma}

\newtheorem*{definition}{Definition}
\newtheorem*{factoid}{Fact}

\def\fqn{\mathbb F_q^n}
\def\ztwo3{\mathbb Z_2^3}
\def\C{{\mathcal C}}
\def\M{{\mathcal M}}

\def\Cal{\mathcal}

\begin{document}

\title[Admissibility for the X-ray transform over $\mathbb Z_2$]{The admissibility theorem for the spatial X-ray transform over the two element field}

\author{ Eric L. Grinberg}
\address{University of Massachusetts Boston }
\email{eric.grinberg@umb.edu }



\begin{abstract}
We consider the Radon transform along lines in an $n$ dimensional vector space over the two element field. It is well known that this transform is injective and highly overdetermined. We classify the minimal collections of lines for which the restricted Radon transform is also injective. This is an instance of I.M.~Gelfand's {\it admissibility problem}. The solution is in stark contrast to the more uniform cases of the affine hyperplane transform and the projective line transform, which are addressed in other papers, \cite{Feld-G,Gr1}.
The presentation here is intended to be widely accessible, requiring minimum background.
\end{abstract}

\maketitle

\section{Introduction}
\label{sec:1}

\subsection{Dedication and two Mathematical Moments} 
\bigskip

This paper is dedicated to the memory of Leon Ehrenpreis. His colleagues were fortunate to have countlessly many discussions with him, after seminars (and during), in offices, hallways, and at the lounge blackboard. These served to inspire, energize and generate many new ideas. The subject of this essay may well have come up in one of these chats. 

During graduate school I learned about the role of {\it spreads} in integral geometry from Ethan Bolker, via an early, handwritten version of \cite{Bol1}  and when I joined Temple University the concept of linear spreads followed and came up in early conversations with Leon. He found spreads to be useful in his approach to integral geometry and he formulated a non-linear variant which he employed in framing his notion of  the {\it non-linear Radon transform}. See the major work \cite{ Ehren} and the review \cite{Ber} by Carlos Berenstein. 

I recall vividly a two-panel chalk board with the level sets of a homogeneous polynomial drawn on one panel, and the heat equation displayed on the other. I also recall sharing a car ride 
with Ethan and Leon, from San Francisco to Arcata, CA, on the way to the 1989 AMS summer conference on integral geometry and tomography, which led to \cite{Gr2-Quinto}. It is safe to say that the majority of the travel time was devoted to an intensive discussion of Radon transforms (and I hope that this did not impair the safety of the ride). The beautiful Calfornia coastline was superceded by admissible line complexes.

\medskip

The structure of spreads (discussed concretely below) is particularly simple in the case of the hyperplane Radon transform over finite fields, and this  can be used to solve the admissibility problem in that context. In contrast, the structure of spreads is not as simple for transforms that integrate over planes of larger codimension, and thus we expect the admissibity problem to have a more complicated solution. Here we investigate the simplest higher-codimension case.

\section{The Radon Transform in a finite geometry.}


The theme of integral geometry, in the style introduced by P.~Funk and J.~Radon and prominent in the work of Leon Ehrenpreis, involves the recovery of functions (or data) from integrals. In applications one might imagine recovering the density distribution of biological tissue from x-ray data. If ``all'' integrals (x-ray) measurements are available then the problem is overdetermined. It is natural to look for minimal sets of data (x-rays) with which complete recovery is still possible (even though in applications such minimal measurements may present stability problems).  Finding and classifying such minimal families is an instance of I.M.~Gelfand's {\it admissibility} problem \cite{Gel2}, which initially occured in the context of the Plancherel formula for semi-simple Lie groups. In the continuous category, the problem depends in part on the choice of function spaces, mapping properties of integral operators, and smoothness properties of collections of lines.  Here we focus on a finite model of integral geometry in which analytic considerations are removed and sets of lines take center stage. In the admissibility theory work of Gelfand and collaborators within the continuous category ($\mathbb R^3$ or $\mathbb C^3$ and their projective and higher dimensional analogs) the family of lines meeting a curve (the Chow variety) and the family of lines tangent to a surface occur as admissible complexes \cite{Gel1, Kir} . Here we will search for finite analogs of these. For discussions of Radon transforms in finite geometries see, e.g., \cite{Kung1, Strich}
Recent results on admissibility in the context of finite projective spaces may be found in \cite{Feld-G}.

Starting with the $q$-element field $\mathbb F_q$ one can build lines, planes, vector spaces of dimension $n$, projective spaces, Grassmannians, and more.  It is easy to use counting measure to define the Radon transform taking functions on $\mathbb F_q^n$ to functions on the set of $k$-planes in $\mathbb F_q^n$: 
$$
R_kf(H) = \sum_{\{x \in H\}} f(x),
$$
where $H$ is a $k$ dimensional affine plane in $\mathbb F_q^n$.
Informally we write
$$
R_k : \, \{ \textrm{point functions} \} \longrightarrow \{ \textrm{k plane  functions} \}
$$
It is natural to  ask: {\it  is the transform $R_k$ invertible?} Rather than answer the question in
this specific case, we consider a more general context, informally borrowing from S.S.~Chern's (1942) formulation of integral geometry \cite{Chern}. 
Consider the following {\it double fibration diagram}:

$$
\begin{xy}
\begin{large}
\xymatrix{
     & Z \ar[dl]_{\pi}\ar[dr]^{\rho} & \\
      X  & & Y}
\end{large}
\end{xy} \qquad\qquad\qquad
$$

\bigskip

\vskip 0.5truein

Chern's formulation was presented in the continuous category; here $X,Y$ and $Z$ are finite sets. We think of $X$ as our space of points, $Y$ as the family of lines or, more generally, submanifolds of $X$, and we think of $Z$ as the {\it incidence manifold} of point-line (or point-generalized line) pairs, so that the point belongs to the line:  
$$
\{ (x,y) \vert x \in y \}.
$$
 The maps
 $\pi$ and $\rho$ are projection functions, e.g., $\pi (x,y)=x$,  so that 
$\pi \times \rho$ is one to one. Thus thinking of $X$ as a set of points and $Y$
as a family of subsets of $X$ is manifested by \cite{Guil-Stern}:
$$
F_y = \pi \circ \rho^{-1} \{ y \}.
$$
When $y$ is a line, $F_y$ is the set of points on the line.
Dually, for every point $x$ we have the set of all subsets $y$  passing
through $x$:

$$
G_x = \rho \circ \pi^{-1} \{ x \}.
$$
With the definitions of $F_y$ and $G_x$ it is possible to relax the condition that $\pi , \rho$ be projections
and consider more general diagrams, though we will not need these here. The double fibration diagram has been
used extensively as a paradigm  for Radon transforms and their generalizations by Gelfand and collaborators, S.~Helgason, 
V.~Guillemin \& S.~Sternberg, and many many others.

A double fibration diagram together with a choice of measures leads to an integral transform. 
In the finite category we will use counting measure and define the notion of Radon transform without
making any additional choices.

Let $C(X),C(Y)$ denote ($\mathbb R$ or $\mathbb C$ valued) functions on set $X,Y$, respectively, and let $f(x), g(y)$ be functions 
in the appropriate spaces; then we define the Radon transform from point functions to line functions by ``integrating" over points
in a line and the dual Radon transform by reversing the role of points and lines:
$$
R: C(X) \longrightarrow C(Y) ; \qquad
R^t : C(Y) \longrightarrow C(X),
$$
$$
Rf(y) \equiv \sum_{\{x \vert x \in y \}} f(x), \qquad
R^tg(x) \equiv \sum_{\{y \vert x \in y\}} g(y).
$$
With $X,Y,Z$ and the double fibration diagram so general can anything be said about invertibility of the induced Radon transform? Surprisingly, 
the answer is affirmative: 

\begin{theorem}[\rm Bolker \cite{Bol1}]  Assume that the double fibration diagram satisfies
the following two conditions:
\begin{itemize}
\item $\# G_x = \alpha$, $\forall x \in X$ \qquad \quad \, (uniform count of lines through each point)
\item $\# G_{x_1} \cap G_{x_2} = \beta$
 $\forall$   $x_1 \ne x_2$ (uniform count of lines through each point pair),
\end{itemize}
for constants $\alpha , \beta$, with  $0 \ne \alpha \ne  \beta$.
Then the Radon transform associated with the diagram is invertible, with an explicit
inversion formula.
\end{theorem}

\noindent
The conditions above, bundled together,  are now known as the {\it Bolker Condition}, which is used extensively.
The proof of the Theorem is straightforward. 

\begin{proof}
We first construct a basis for $C(X)$.
Let $\delta_p$ be the function on $X$ with value $1$ at $p \in X$ and $0$ elsewhere. Let $n$ be the cardinality of $X$.
Then $\{ \delta_x \}_{x \in X}$ is a basis for $C(X)$, which has dimension $n$. There is a similar 
basis for $C(Y)$. The matrix of the composed transform $R^t \circ R$ in 
this basis is:
$$
\left(
\begin{array}{ccc}
\alpha & \cdots & \beta \\
\vdots & \ddots & \vdots \\
\beta  & \cdots & \alpha \\
\end{array}
\right)
= ( \alpha - \beta ) {\mathbf  I} + \beta {\mathbf  1},
$$
where $\mathbf I , \mathbf 1$ denote the $n \times n$ identity matrix, 
and the $n \times n$ matrix with $1$'s in every entry, resp. To invert, 
note that
$
\displaystyle
(a{\mathbf  I} +  b{\mathbf  1}) 
\cdot 
(c{\mathbf  I} +  d{\mathbf  1}) 
= (ac) {\mathbf  I} +  (ad+bc+nbd){\mathbf  1}.
$
 \end{proof}
The Bolker condition is satisfied by many geometric
double fibrations, but does not hold for many others,
even when the Radon transform is injective.
The Radon transform on a triangle, rectangle, and 
pentagon can be represented by the following  matrics: 

\begin{figure}[h]
\centerline{
\includegraphics[scale=0.8]{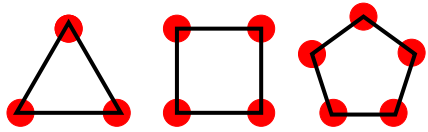}}
\caption{\small Some finite geometries with and without the Bolker condition,
and corresponding matrices}
\end{figure}

\begin{displaymath}
\qquad \quad
\left(
\begin{array}{ccc}
1 & 1 & 0 \\
0 & 1 & 1 \\
1 & 0 & 1
\end{array}
\right)
\left(
\begin{array}{cccc}
1 & 1 & 0 & 0 \\
0 & 1 & 1 & 0 \\
0 & 0 & 1 & 1 \\
1 & 0 & 0 & 1 \\
\end{array}
\right)
\left(
\begin{array}{ccccc}
1 & 1 & 0 & 0 & 0 \\
0 & 1 & 1 & 0 & 0 \\
0 & 0 & 1 & 1 & 0 \\
0 & 0 & 0 & 1 & 1 \\
1 & 0 & 0 & 0 & 1 \\
\end{array}
\right)
\end{displaymath}

\bigskip

It is easy to verify the properties in the table below for the Radon transform on these geometries.

\bigskip

\centerline{
\begin{tabular}{|c|c|c|}
\hline
  $\#$ sides & Bolker C. Satisfied?& R injective?  \\
\hline
         3 & Yes & Yes  \\
\hline
         4 & No & No \\
\hline
         5 & No & Yes \\
\hline
\end{tabular}
}
\bigskip

The $k$-plane transform in $\mathbb F_q^n$ satisfies 
the Bolker Condition, since given points $x_1,x_2$ there is an affine map $T$
that carries $x_1$ to $x_2$ and the set of lines through $x_1$ to the
set of lines through $x_2$, and given two pairs of points, there is an affine map that
carries one pair to the other and the lines through one pair to the lines through the other pair.
More generally, the Bolker Condition holds whenever there is a doubly transitive group action
that preserves the appropriate incidence relations. When group symmetry is available it is natural
consider the use of group representations. Interestingly, representation theory can be used to understand Radon transforms on the one hand, e.g., \cite{Zel}, and Radon transforms can be used to understand representation theory, e.g., \cite{ Guil-Stern, Stern}.

We may also inquire about a range characterization: when is a function of $k$-planes the Radon transform of a function
of points? We first look at the hyperplane case, $k=n-1$.
\begin{definition} A {\it spread} of hyperplanes in $\mathbb F_q^n$
is a presentation of $\fqn$ as a disjoint union of hyperplanes.
\end{definition}


\begin{factoid}
 A function $g(H)$ of hyperplanes $H$ in $\fqn$ is
the Radon transform of a function of points $x\in \fqn$ {\bf only if}
the average of $g(H)$ over any spread is the same as the average 
over any other spread:
$$
\sum_{ \{ H \in \Omega_1 \} } g(H) 
= 
 \sum_{\{ H \in \Omega_2\}} g(H)
\eqno{\textrm{ (for any two spreads $\Omega_1, \Omega_2$). }}
$$
\end{factoid}
These are called the {\it Cavalieri conditions}.
By way of illustration, In the diagram below, they state that the sum over lines with positive slope equals the sum over lines with negative slope.
\begin{figure}[h]
\centerline{
\includegraphics[scale=0.62]{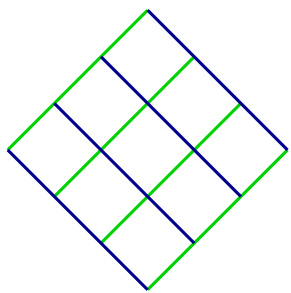}}
\caption{Two spreads leading to a Cavalieri condition}
\end{figure}

\begin{theorem}[\rm Bolker] The Cavalieri conditions characterize
the range of the hyperplane Radon transform over a finite field.
\end{theorem}

\noindent
The proof  is based on a counting argument.
This range condition yields an {\it admissibility} theorem.

\section{Admissible Complexes}

\begin{definition}
Recall that a {\bf complex} of hyperplanes $\C$ is a collection
of hyperplanes $\{ H \vert H \in \C \}$ so that $\# \C = \# \fqn =q^n$
(there are as many hyperplanes as points). We'll also use ``complex" to
denote the appropriate number of lines, curves, etc.
\end{definition}
\begin{definition}
The complex $\C$ is said 
to be {\bf admissible} if the Radon transform operation, restricted to
planes belonging to $\C$ is still injective:
$$
R_\C : C(\fqn ) \longrightarrow C(\C ).
$$
\end{definition}

\begin{theorem}{( \cite{Gr1})} A complex $\C$ of hyperplanes in $\fqn$ is
admissible  if and only if it omits precisely one plane from each spread,
except for one spread, which belongs to $\C$ in its entirety.
\end{theorem}

\noindent
To prove ``if", it suffices to show that $R_\C f$  determines $Rf$. 
A counting argument shows that every complex contains an entire spread.
To evaluate $Rf$ on a plane $H$ which $\C$ omits, simply use the 
total mass of $f$ encoded in a spread that belongs to $\C$ in
its entirety.  To prove ``only if", take two parallel hyperplanes 
and construct a ``capacitor" charge distribution: $+1$ on plane,
$-1$ on the other, and zero elsewhere. Only the two chosen planes
can ``see" this distribution via the Radon transform. The rest have 
vanishing Radon transform because of cancellation.

\noindent
Thus the hyperplane case turns out to be the easy case. We now explore 
the next simplest: the line transform in $\mathbb Z_2^3$.
The three dimensional vector space over $\mathbb Z_2$ has
 $8$ points,
 $7$ lines through a given point, $28$ lines in all.

\begin{figure}[h!]
\centerline{
\includegraphics[scale=0.99]{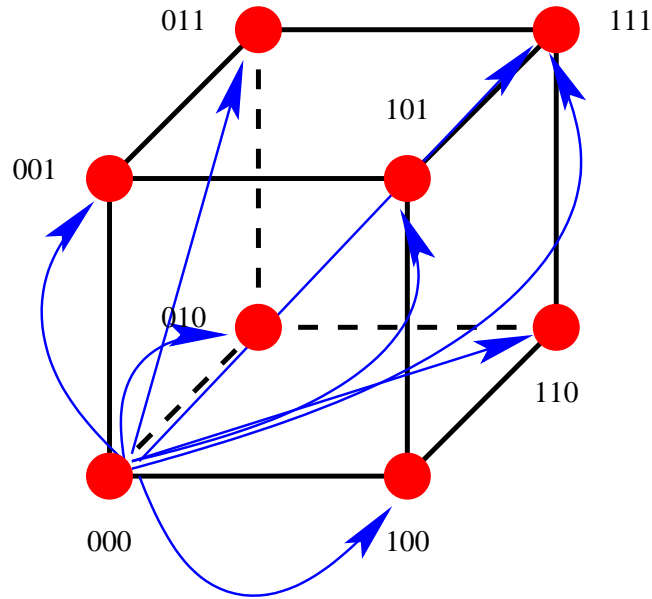}}
\caption{lines in $\mathbb Z_2^3$}
\end{figure}

\newpage

Here are some ways to construct admissible complexes:

\begin{itemize}

\item Write $\ztwo3$ as a union of two parallel planes (a {\it spread} of planes)
 and choose an admissible set of lines on each plane (four lines chosen in each
 plane). 

\bigskip

\item Choose one plane $\P \subset \ztwo3$, choose an admissible set of (four)
 lines within $\P$, then extend four {\it ``legs"} perpendicular to $\P$. 

\item Construct, if possible, admissible complexes in $\ztwo3$ without using
 planar relatively admissible complexes.
\end{itemize}

\noindent 
The first two methods are illustrated below.

\bigskip\bigskip

\begin{figure}[h!]
\centerline{
\includegraphics[scale=0.50]{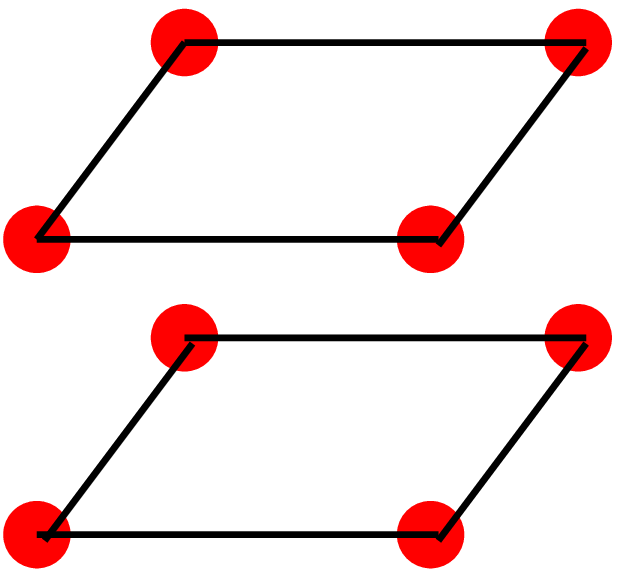} 
\qquad \qquad \qquad \qquad 
\includegraphics[scale=0.50]{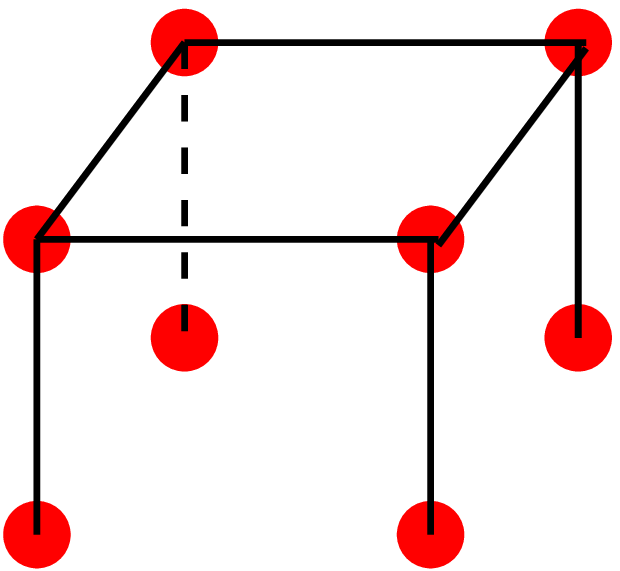}
}
\caption{Some ways to construct admissible complexes.}
\end{figure}

\bigskip 

\par\vfill\eject

\noindent
The Radon transform for lines in $\mathbb Z_2^3$ can be 
represented by the following $28 \times 8$ matrix:

\begin{displaymath}
\left(
\begin{array}{cccccccc}
1 & 1 & 0 & 0 & 0 & 0 & 0 & 0 \\
1 & 0 & 1 & 0 & 0 & 0 & 0 & 0 \\
1 & 0 & 0 & 1 & 0 & 0 & 0 & 0 \\
1 & 0 & 0 & 0 & 1 & 0 & 0 & 0 \\
1 & 0 & 0 & 0 & 0 & 1 & 0 & 0 \\
1 & 0 & 0 & 0 & 0 & 0 & 1 & 0 \\
1 & 0 & 0 & 0 & 0 & 0 & 0 & 1 \\
0 & 1 & 1 & 0 & 0 & 0 & 0 & 0 \\
0 & 1 & 0 & 1 & 0 & 0 & 0 & 0 \\
0 & 1 & 0 & 0 & 1 & 0 & 0 & 0 \\
0 & 1 & 0 & 0 & 0 & 1 & 0 & 0 \\
0 & 1 & 0 & 0 & 0 & 0 & 1 & 0 \\
0 & 1 & 0 & 0 & 0 & 0 & 0 & 1 \\
0 & 0 & 1 & 1 & 0 & 0 & 0 & 0 \\
0 & 0 & 1 & 0 & 0 & 0 & 0 & 0 \\
0 & 0 & 1 & 0 & 1 & 0 & 0 & 0 \\
0 & 0 & 1 & 0 & 0 & 1 & 0 & 0 \\
0 & 0 & 1 & 0 & 0 & 0 & 1 & 0 \\
0 & 0 & 1 & 0 & 0 & 0 & 0 & 1 \\
0 & 0 & 0 & 1 & 1 & 0 & 0 & 0 \\
0 & 0 & 0 & 1 & 0 & 1 & 0 & 0 \\
0 & 0 & 0 & 1 & 0 & 0 & 1 & 0 \\
0 & 0 & 0 & 1 & 0 & 0 & 0 & 1 \\
0 & 0 & 0 & 0 & 1 & 1 & 0 & 0 \\
0 & 0 & 0 & 0 & 1 & 0 & 1 & 0 \\
0 & 0 & 0 & 0 & 1 & 0 & 0 & 1 \\
0 & 0 & 0 & 0 & 0 & 1 & 1 & 0 \\
0 & 0 & 0 & 0 & 0 & 1 & 0 & 1 \\
0 & 0 & 0 & 0 & 0 & 0 & 1 & 1 \\
\end{array}
\right)
\end{displaymath}

\medskip

\noindent
The admissibility problems asks: 

\medskip

\centerline{\it
What are the non-singular $8 \times 8$ minors of this matrix?}

\medskip

\noindent
We'd like an answer that is geometrically motivated. The linear algebra computer environment {\tt Octave} can be used to locate all admissible complexes for this transform, as illustrated below.


\newpage

\centerline{
%
\scalebox{0.9}{
\frame{\includegraphics[width=0.8\textwidth]
{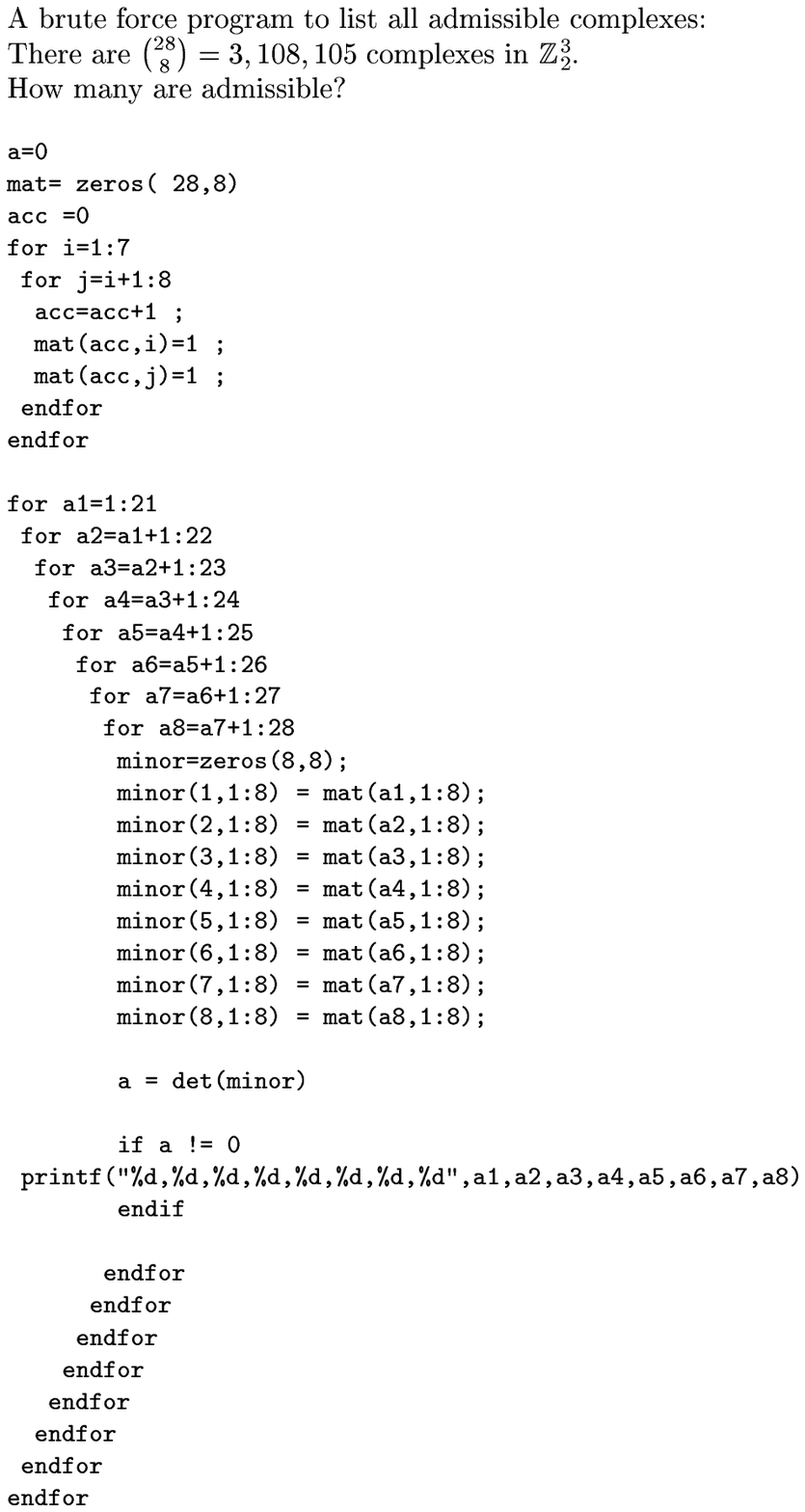}}
}
}
The program results give:
\begin{itemize}

\item 3,108,105 line complexes
\item 2,170,667 inadmissible complexes {\tt\tiny [later corrected to 2,170,665]}
\item 937,438 admissible complexes {\tt\tiny [later corrected to 937,440]}

\end{itemize}

Can we describe the {\it moduli space} of admissible complexes?
Can we enumerate them without using brute force?

There are some clear obstructions to admissibility. In particular, a complex $\C$ is inadmissible if it has any of the following features.

\begin{itemize}

\item An omitted point.
\item An isolated tree.
\item An even cycle.

\end{itemize}
Clearly, a line complex that does not pass through a particular point cannot recover data at that point. Similarly, complexes with even cycles or with isolated trees are rank-deficient, as manifested  by a $+1,-1$ data pattern. These contexts are illustrated below.

\begin{figure}[h]
\centerline{
\scalebox{0.25}
{
\includegraphics{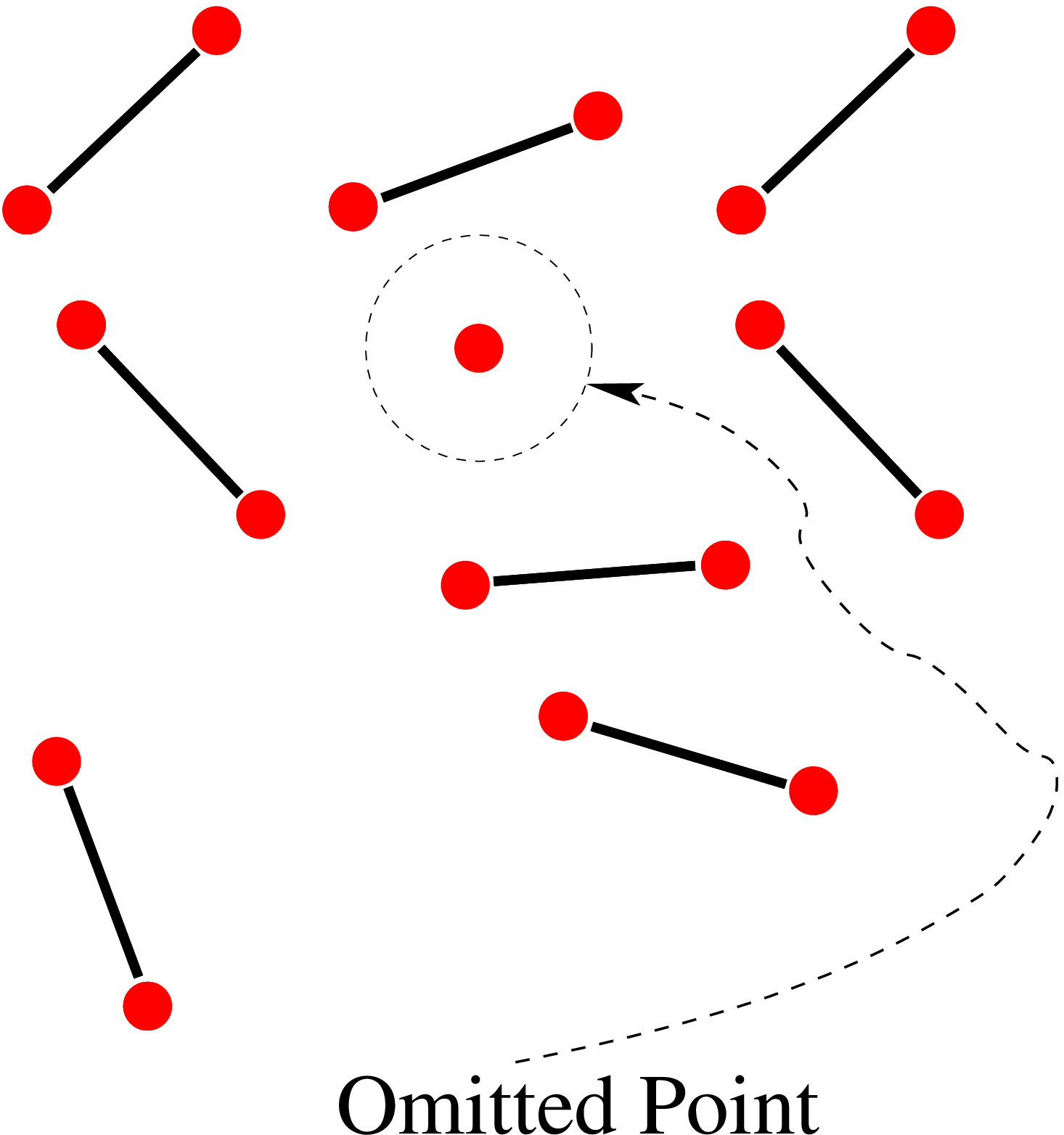} \qquad \qquad \qquad \qquad \qquad \includegraphics{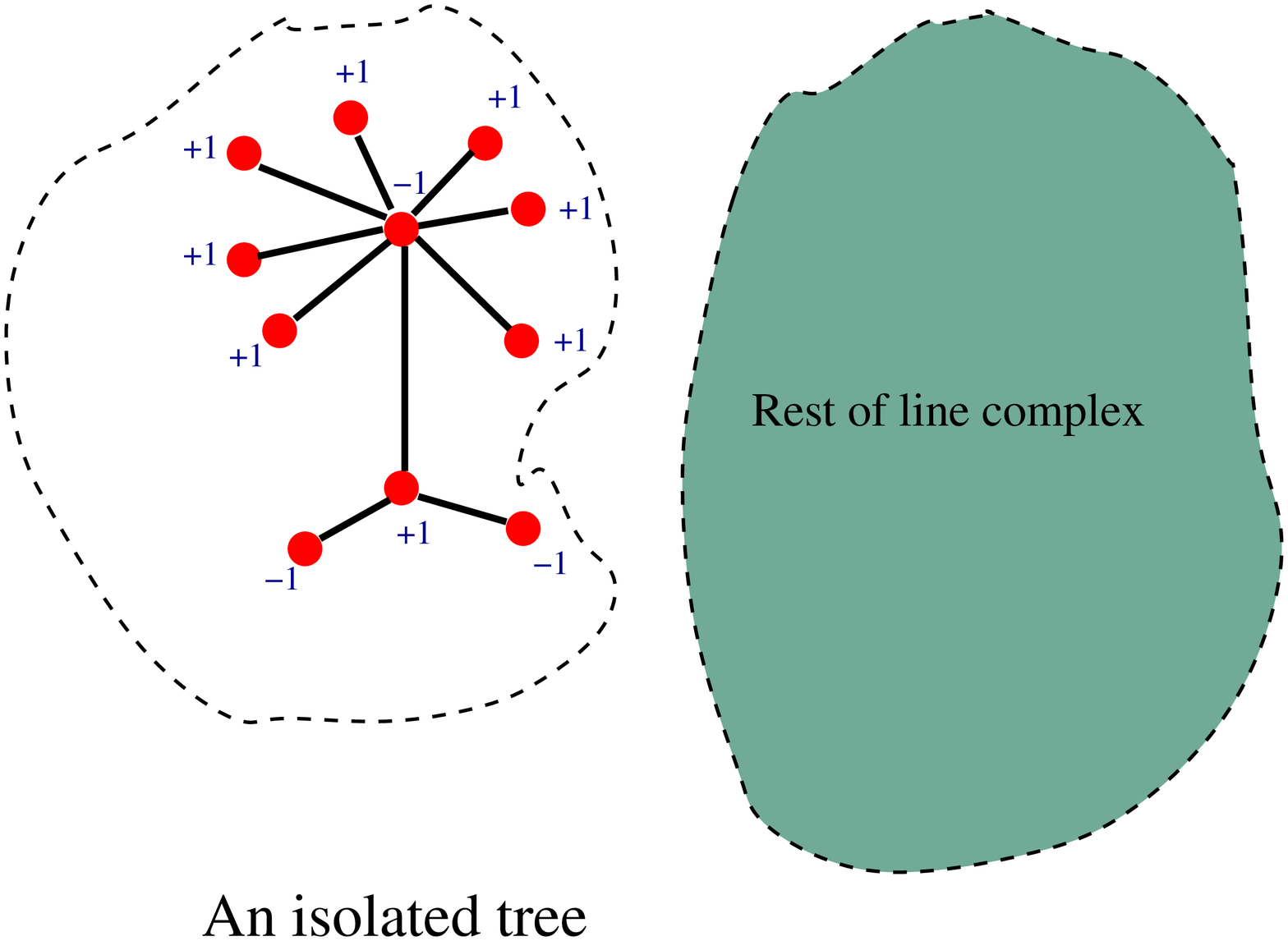} 
}}

\bigskip

{\scalebox{0.25}{
\includegraphics{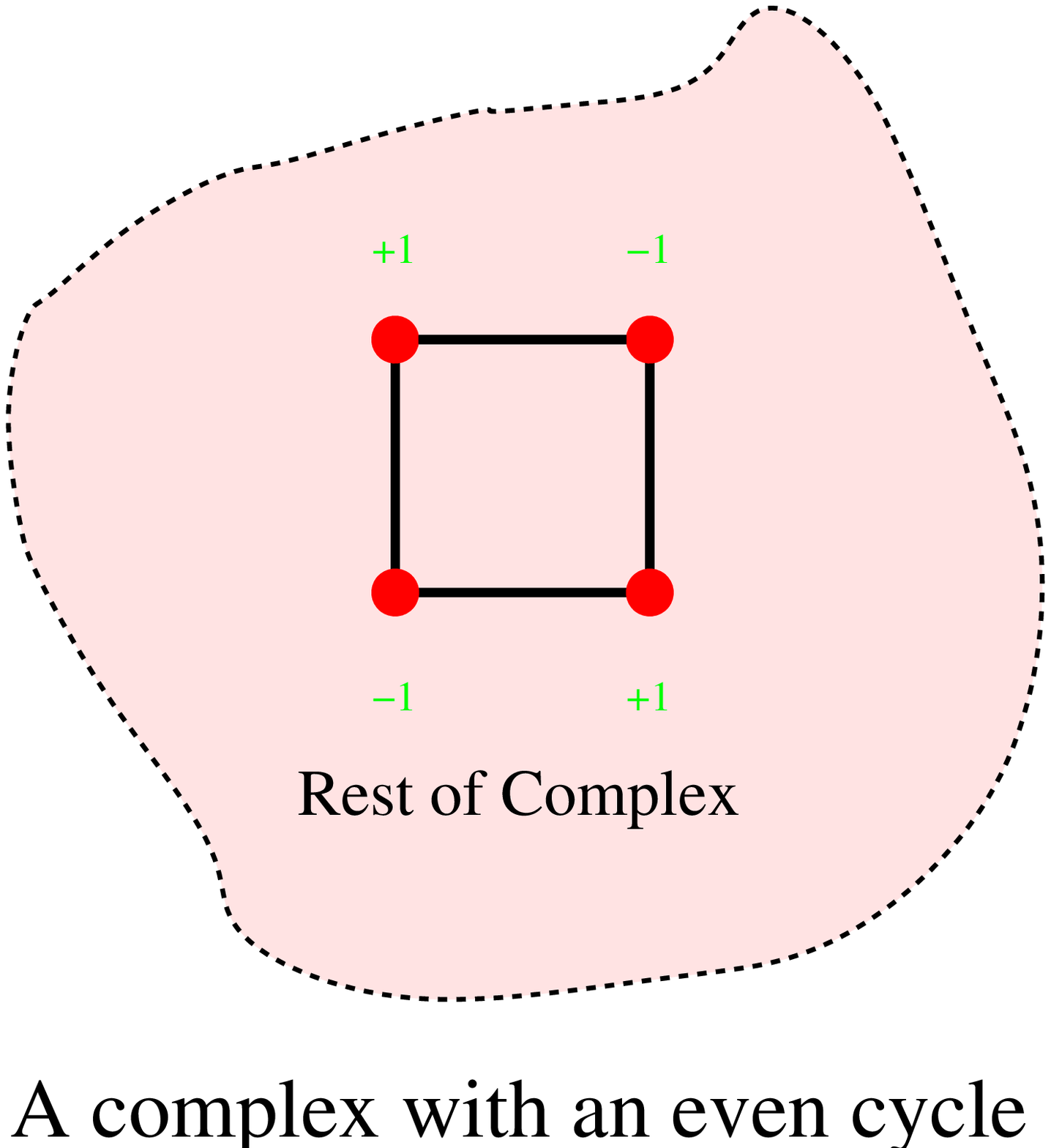}
}}
\caption{Some inadmissible configurations}
\end{figure}

\noindent
It turns out that these are {\it the only } obstructions to admissibility.

\begin{theorem} [Admissibility for $\mathbb Z_2^n$]
Let $\C$ be a line complex in $\mathbb Z_2^n$. Assume that $\C$ omits no point,
 has no isolated trees, and does not contain an even cycle. Then $\C$ is
 admissible.
\end{theorem}

\begin{proof}
Take a point $p \in {\mathbb Z}_2^n$. There's a line $\ell \in \C$
 containing $P$. Expand $\ell$ to a maximal connected set of lines, $\M$.  Then
 $\M$ cannot be a tree, so $\M$ contains cycles, hence odd cycles. Each odd cycle
 is ``self inverting". Every point in $\M$ is linked to an odd cycle by a
 contiguous path of lines, hence is solvable.  
\end{proof}

\newpage

\section{Appendix: Counting a majority of inadmissible complexes}
Here we will count two basic archetypes of inadmissible complexes, along with their intersection. This will serve to illustrate the combinatorics of the complete count. 

\subsection{Complexes that omit one or more points.}

First we enumerate complexes that are ``missing points'', that is, complexes $\Cal C$ so that there exist points $p \in \mathbb F_2^3$ so that no line $\ell \in \Cal C$ passes through $p$.
It turns out that there are many of these. There are seven lines through $p$, so the complexes that miss $p$ have $8$ lines chosen from the $28-7=21$. Now $\binom{21}{8}=203,440$. Multiplying this by the number of points, $8$, and accounting for  double counting (because there are complexes that omit more than one point) we obtain:

\begin{lemma} There are $\binom{21}{8} \times 8 = 1,627,920$ complexes that omit points. Here, each complex is counted with multiplicity equal to the number of points in $\mathbb F_2^3$ which it misses. \end{lemma}

\subsubsection{Complexes that omit two or more points.}

How many complexes miss two points? There are $7+7-1=13$ lines through one or the other or both points. So a complex that misses both points has $8$ lines chosen from among $28-13=15$ lines. There are $28$ pairs of points, so we have double counted $28 \times \binom{15}{8}=28 \times 6,435
= 180,180$ complexes. (Note that we have double counted the double counting, because there are complexes that miss three points.)

\begin{lemma} The number of complexes that omit a pair of points is  
$28 \times \binom{15}{8}=28 \times 6,435 = 180,180.$ Here each complex is counted with multiplicity equal to the number of pairs of points that it misses.
\end{lemma}

\subsubsection{Complexes that omit three or more points.}

How many lines pass through one or more of three given 
points? All but the $10$ that form the complete graph on the remaining $5$ points. Thus, to exhibit all complexes omitting three or more points, choose three points from $8$ and then choose $8$ lines from among $10$. Thus we have: 

\begin{lemma} The number of complexes that omit precisely three points is
$\binom{10}{8} \times \binom{8}{3} = 2,520.$ There are no line complexes that miss four or more points.
\end{lemma}

Putting the above lemmas together we have

\begin{lemma}
The number of complexes that avoid one or more points is:
$1,627,920 - 180,180 + 2,520 = 1,450,260.$ This count is without multiplicity.
\end{lemma}

\subsection{Complexes with isolated lines.}

\subsubsection{Complexes with one or more isolated lines.}

Another type of non-admissible complex is one where a single line $\ell$ is `isolated', i.e., meets no other line in the complex. (This is the simplest case of an isolated tree.)
How many of these are there? Well, how many lines meet $\ell$? $7+7-1=13=28-15$. So the number of complexes having $\ell$ as an isolated line is $\binom{15}{7}= 6,435.
$ Accounting for each of $28$ lines, with the usual double counting reminder, we have 

\begin{lemma} There are
$6,435 \times 28= 180,180 $ complexes with one or more isolated lines. 
Each complex is counted with multiplicity equal to the number of isolated lines it has.
\end{lemma}

\subsubsection{Complexes with two or more disjoint isolated lines.}

If $\ell$ is a line, there are $13$ lines meeting $\ell$ and $15$ lines disjoint from $l$. Thus there are $(28)(15)/2=210$ pairs of disjoint lines. Given a complex with a pair of disjoint lines, the other $6$ lines of the complex must form the complete graph on the remaining four points. Thus there are $210$ complexes with precisely two disjoint isolated lines. Clearly a complex cannot have three disjoint isolated lines.

\begin{lemma} 
There are $(28)(15)/2=210$ complexes with precisely two isolated lines, and there are no complexes with three or more isolated lines.
\end{lemma}

\begin{lemma}
There are $180,180-210=179,970$ complexes with one or more isolated lines. These complexes are counted without multiplicity.
\end{lemma}

\subsection{Complexes with both omitted points and isolated lines.}
\subsubsection{Complexes with one or more isolated lines and one or more omitted points.}

There are five points disjoint from the designated omitted point and the isolated line, hence there are $\binom{5}{2}=10$ permissibile lines. We must choose $7$ lines among these to form a complex, and there are $8 \times 28$ point, line pairs.

\begin{lemma} There are no  complexes with one isolated line and two omitted points. 
\end{lemma}

\begin{proof}
The complement of the union of the omitted points and the isolated line has $4$ points, and these form $6$ lines, not enough to form a line complex. 
\end{proof}

\begin{lemma}
There are no complexes with two disjoint isolated lines and an omitted point
\end{lemma}

\begin{proof}
There are five points in the union of the two lines and point, hence three points left, not enough to span a line complex. 
\end{proof}

\begin{lemma}
The number of complexes with one isolated line and one omitted point is 
$(8 \times 21) \binom{10}{7}=20,160$.  The count is multiplicity free.
\end{lemma}

\begin{proof}
There are $8 \times 21=168$ disjoint point-line pairs (or $28 \times 6=168$ disjoint line-point pairs). Given a disjoint point-line pair there are $5$ remaining points and $\binom{5}{2}=10$ lines in their complete graph. Of these we must choose $7$ to obtain  a line complex. Because of the preceeding lemmas there are no multiplicities. Hence the claimed count is verified. 
\end{proof}

\medskip
We have counted a majority of inadmissible complexes and illustrated the combinatorics of intersections of archetypes. If sufficient interest develops we will post a completion of this analysis. \medskip

\noindent
{\emph{Added in proof:} 

\noindent
\scalebox{0.8}{\tt  This analysis is now completed and included in a follow-up paper with Mehmet Orhon.}}

\section{acknowledgement}
The author thanks the referee of another version of the paper for helpful comments and suggestions.

\end{document}